\documentclass[reqno]{amsart}
\usepackage[numbers,square,comma,sort&compress]{natbib}
\usepackage{amsmath,amssymb}
\usepackage{color}

\usepackage{perpage} 
\MakePerPage{footnote}

\newcommand{\real}{\mathbb{R}}

\newcommand{\complex}{\mathbb{C}}

\newcommand{\rn}{\real^N}

\newcommand{\intr}{\int_{\real}}
\newcommand{\intoi}{\int_0^\infty}

\newcommand{\eps}{\varepsilon}
\newcommand{\ep}{\epsilon}

\newcommand{\al}{\alpha}
\newcommand{\ffi}{\varphi}

\newcommand{\lam}{\lambda}
\newcommand{\ly}{\lambda_\infty}
\newcommand{\ls}{\lambda_*}

\newcommand{\e}{\mathrm{e}}

\newcommand{\la}{\langle}
\newcommand{\ra}{\rangle}
\newcommand{\wt}{\widetilde}
\newcommand{\diff}{\,\mathrm{d}}
\newcommand{\dif}{\mathrm{d}}
\newcommand{\Ve}{\Vert}
\newcommand{\ve}{\vert}
\renewcommand\emptyset{\mbox{\Large \o}}
\newcommand{\sm}{\setminus}
\newcommand{\x}{\times}
\newcommand{\les}{\leqslant}
\newcommand{\ges}{\geqslant}
\newcommand{\p}{\partial}

\newcommand{\disp}{\displaystyle}

\DeclareMathOperator \re{Re}

\newtheorem{theorem}{Theorem}[section]
\newtheorem{lemma}[theorem]{Lemma}
\newtheorem{proposition}[theorem]{Proposition}

\newtheorem{definition}[theorem]{Definition}
\newtheorem{remark}[theorem]{Remark}
\newtheorem{example}[theorem]{Example}

\numberwithin{equation}{section}

\begin{document}

\title[Orbitally stable standing waves]
{Orbitally stable standing waves for the asymptotically linear 
one-dimensional NLS}
\author{Fran\c cois Genoud}

\date{20 November 2012}
\thanks{This work was supported by the Engineering and Physical Sciences 
Research Council, [EP/H030514/1].}


\address{Department of Mathematics and the Maxwell Institute
for Mathematical Sciences \\ Heriot-Watt University \\
Edinburgh \\ EH14 4AS \\ Scotland.}
\email{F.Genoud@hw.ac.uk}

\begin{abstract}
In this article we study the one-dimensional, asymptotically linear, 
non-linear Schr\"odinger equation (NLS). We show the existence of a global
smooth curve of standing waves for this problem, and we prove that these
standing waves are orbitally stable. As far as we know, 
this is the first rigorous stability 
result for the asymptotically linear NLS. We also discuss an application
of our results to self-focusing waveguides with a saturable refractive index.
\end{abstract}

\maketitle

\section{Introduction}

In this article we study the one-dimensional 
nonlinear Schr\"odinger equation
\begin{equation}\label{NLS}\tag{NLS}
i\partial_t\psi + \partial^2_{xx}\psi + f(x,|\psi|^2)\psi = 0
\end{equation}
for $\psi=\psi(t,x):[0,\infty)\times\real\to\complex$.
We suppose that $f\in C^1(\real\times\real_+,\real)$, and satisfies 
the following assumptions.
\medskip
\begin{itemize}
\item[\bf(A0)] $f(x,0)=0$ for all $x\in\real$, and we have
\begin{align}\label{f0as}
\lim_{s\to0}f(x,s)&=0, \ \text{uniformly for $x\in\real$},\\
\lim_{|x|\to\infty}f(x,s)&=0, \ \text{uniformly for $s\ges0$}.\label{finfas}
\end{align}
\item[\bf(AL)]
There exists 
$f_\infty \in C(\real)\cap L^\infty(\real)$ such that
\begin{equation}\label{as.eq}
\lim_{s\to\infty}f(x,s)=f_\infty(x), \ \text{uniformly for} \ x\in\real.
\end{equation}
\end{itemize}

Equation \eqref{NLS} with the assumption (AL) is usually referred to as the 
{\em asymptotically linear} nonlinear Schr\"odinger equation.
The structure of the nonlinearity allows one to look for 
{\em standing wave} solutions 
$$
\psi(t,x)=\e^{i\lam t}u(x), \quad
\text{where} \quad \lam>0 \quad \text{and} \quad u:\real\to\real.
$$
As we shall see from the stability analysis in Section~\ref{stability.sec},
solutions of this type enjoy remarkable properties with respect to the
general dynamics of \eqref{NLS}.
Using this Ansatz to solve \eqref{NLS} yields the stationary
equation
\begin{equation}\label{stat}\tag{SNLS}
u'' + f(x,u^2)u = \lambda u.
\end{equation}
This second order ordinary differential equation will be interpreted as 
a nonlinear eigenvalue problem, which will be addressed via bifurcation theory.

By a solution to \eqref{stat} will be meant a couple 
$(\lam,u)\in \real\x H^1(\real)$ satisfying \eqref{stat} in the sense
of distribution. 
Note that, since $f(x,0)\equiv0$, we
have a line of {\em trivial solutions}, 
$\{(\lam,0):\lam\in\real\}\subset\real\x H^1(\real)$.
We will prove that there exist non-trivial solutions,
bifurcating from the line of trivial solutions.
Under appropriate hypotheses on $f$, in particular
assuming that $f(x,s)$ is positive and even in $x$,   
we will obtain a smooth curve of positive even solutions of \eqref{stat},
\begin{equation}\label{curve}
\{(\lam,u(\lam)):\lam\in(0,\ly)\}\subset\real\x H^1(\real).
\end{equation}
The number $\ly\in(0,\infty)$ will be characterized as 
the principal eigenvalue of the linear problem
\begin{equation}\label{lin}
u''+f_\infty(x)u=\lambda u,
\end{equation}
known as the {\em asymptotic linearization}.

Additional properties of the solutions will be proved, in particular their
bifurcation behaviour:
$$
\lim_{\lam\to0}\Ve u(\lam)\Ve_{H^1(\real)}=0 \quad\text{and}\quad
\lim_{\lam\to\ly}\Ve u(\lam)\Ve_{H^1(\real)}=\infty.
$$
The curve \eqref{curve} will be obtained, first by a local bifurcation
analysis near $u=0$, then by analytic continuation and an asymptotic analysis
as $\Ve u(\lam)\Ve_{H^1}\to\infty$. Bifurcation
from the line of trivial solutions is difficult in the present context since,
under assumption \eqref{f0as}, the linearization of \eqref{stat} at $u=0$ is
\begin{equation}\label{linat0}
u''=\lam u,
\end{equation}
and has {\em purely continuous spectrum}. 

In Subsection~\ref{start.sec}, prescribing the behaviour of $f(x,s)$ as
$s\to0$, we will obtain bifurcation of 
{\em positive} solutions of \eqref{stat}, from the bottom of
the continuous spectrum of \eqref{linat0} (i.e. from $\lam=0$), 
by applying a fairly involved perturbation analysis, introduced in \cite{g1}.
This local result is contained in Theorem~\ref{local.thm}.

Under additional assumptions (in particular
assuming that $f(x,s)$ is positive and even in $x\in\real$),
a global analysis is carried out in Subsection~\ref{global.sec}, where we
show that the local branch given by Theorem~\ref{local.thm} can be
extended in a smooth manner to the curve \eqref{curve}. The global
continuation relies on the non-degeneracy of positive even solutions of
\eqref{stat} given by Lemma~\ref{nondegen.lem}. The asymptotic behaviour
of the branch as $\Ve u(\lam)\Ve_{H^1}\to\infty$ (and $\lam\to\ly$)
follows by the asymptotic bifurcation analysis that was developped in 
\cite{g3} using topological arguments. Our global result is 
Theorem~\ref{global.thm}.	

In \cite{g3} (where we only considered
\eqref{stat} for $x>0$) we obtained global bifurcation
from $(\ly,\infty)$ in $\real\times H^2(0,\infty)$. More precisely, 
we obtained a global connected set of solutions $(\lam,u)$ such that
$\Ve u\Ve_{H^2}\to\infty$ as $\lam\to\ly$. This
was later extended in \cite{g4} to a higher dimensional
version of \eqref{stat}, under fairly weak hypotheses,
providing a strong existence result for positive solutions of
\eqref{stat} with assumption (AL). Previous contributions on 
this problem --- see e.g. \cite{ct,j,jt,sz,sz96,sz99,sz05,zz} and the references in these
papers --- mostly used variational methods to prove existence of solutions 
for \eqref{stat}-(AL), under various assumptions, typically stronger
than those required by the topological approach. 

However, due to technical restrictions of the method used 
in \cite{g3}, we were not able to continue the bifurcating branch down
to the line of trivial solutions. Section~\ref{bifurcation.sec} below 
completes the discussion initiated in \cite{g3} by showing that, 
under more restrictive assumptions (in particular, symmetry and monotonicity
conditions on $f(x,s)$, as well as a precise asymptotic behaviour as $s\to0$),
there exists a smooth curve of solutions, connecting $(0,0)$ to $(\ly,\infty)$
in $\real\times H^1(\real)$.

A global solution curve was obtained by Jeanjean and Stuart \cite{js}
under similar hypotheses. However, they consider an
additional, non-trivial, linear potential in \eqref{stat}. In our notation,
this amounts
to assuming that $f_0(x):= f(x,0)\not\equiv0$ instead of $f(x,0)\equiv0$
in (A0). In this case,
the linearization at $u=0$ has the form
\begin{equation}\label{linat0'}
u''+f_0(x)u=\lam u.
\end{equation}
Under appropriate assumptions on $f_0$ (e.g. $f_0$ is a `bump'), the linear 
problem \eqref{linat0'} has a principal eigenvalue, from which bifurcation can
be obtained via standard bifurcation theory. Global continuation can then be 
obtained by arguments similar to those of Subsection~\ref{global.sec}. The
authors of \cite{js} also discussed the case of an asymptotically linear
nonlinearity, as in \eqref{as.eq}. The main difference in the present context
is that bifurcation at $u=0$ occurs from the bottom of the continous spectrum
of the linearization \eqref{linat0}.

In Section~\ref{stability.sec} we will consider the
standing wave solutions 
$\psi_\lam(t,x)=\e^{i\lam t}u(\lam)(x)$ of \eqref{NLS} corresponding to
the solutions \eqref{curve} of \eqref{stat}. 
We will prove that they are {\em orbitally stable} amongst 
the set of solutions
$\psi(t,x)\in C\big([0,\infty),H^1(\real,\complex)\big)$. This result,
Theorem~\ref{stability.thm}, is based on the general theory of orbital
stability for Hamiltonian systems, see \cite{gss,stuart2008}. Given a
standing wave $\psi_{\lam_0}$, it follows from the theory that,
under
appropriate conditions on the spectrum of the linearization of \eqref{NLS}
at $\psi_{\lam_0}$, this standing wave is orbitally stable
if the mapping $\lam\to\Vert u(\lam)\Vert_{L^2}^2$ is increasing 
at the point $\lam=\lam_0$. Since we have a smooth 
curve of solutions, this can be obtained by checking that
\begin{equation}\label{slope.eq}
\frac{\dif}{\dif\lam}\Big|_{\lam=\lam_0}\Vert u(\lam)\Vert_{L^2}^2>0.
\end{equation}
Our bifurcation analysis in Section~\ref{bifurcation.sec}
shows that $\Vert u(\lam)\Vert_{L^2}\to0$ as $\lam\to0$, and so \eqref{slope.eq}
must hold for some $\lam_0\in(0,\ly)$. 
Hence, by continuity, we need only check that
$$
\frac{\dif}{\dif\lam}\Vert u(\lam)\Vert_{L^2}^2\neq0, \quad\text{for all} \ 
\lam\in(0,\ly).
$$
This is done in Subsection~\ref{slope.sec}, using an 
integral identity that was first derived in \cite{mst}
to study orbital stability along
the solution curve obtained in \cite{js}. However, the authors 
of \cite{mst} were not able to deal with the asymptotically linear case.
In fact, a careful inspection of the proof of \cite[Theorem~2.1]{mst}
shows that the non-trivial potential $f_0$ in \eqref{linat0'}
obstructs the argument under assumption \eqref{as.eq}. In the
context of (A0) (i.e. with $f_0\equiv0$), we can prove that
condition \eqref{slope.eq} is verified under assumption \eqref{as.eq}, for all
$\lam_0\in(0,\ly)$. This result
(Proposition~\ref{slope.prop}) and the spectral conditions
(Proposition~\ref{spectral.prop}) yield the stability of the
standing waves $\psi_\lam(t,x)=\e^{i\lam t}u(\lam)(x)$, for all
$\lam\in(0,\ly)$ (Theorem~\ref{stability.thm}). 
To the best of our knowledge, Theorem~\ref{stability.thm}
is the first rigorous stability result for the asymptotically linear
Schr\"odinger equation.	

Lastly, Section~\ref{physics} is devoted to an application in nonlinear
optics. Following our study in \cite{g2} of self-focusing planar waveguides 
in Kerr media, the results of Sections~\ref{bifurcation.sec} and
\ref{stability.sec} under assumption (AL) 
now allow us to discuss the existence and stability
of TE travelling waves in materials having a saturable
dielectric response.

\subsection{Open problems}
Various problems remain unsolved in the higher
dimensional setting. For instance, the global analytic continuation carried out
in Subsection \ref{global.sec} makes use of the uniqueness of positive even solutions to \eqref{stat}
(for each fixed $\lam$). As far as we know, the uniqueness problem in higher dimension is still open.
Moreover, due to a lack of compactness coming from the unboundedness of the domain,
the analytic continuation seems hard to obtain without uniqueness.

In dimension $N\ges3$, 
the orbital stability of standing waves along a {\em local} branch of solutions
--- such as that obtained in Theorem~\ref{local.thm} --- follows from
\cite[Theorem~1~(b)]{g1}. Nonetheless, even if a smooth {\em global} 
curve of positive radial solutions existed in the radial case,
our method to prove stability along the whole curve might not work 
in dimension $N>1$ because of an extra term of the form
$\tfrac{N-1}{r}u'(r)$ in the radial (higher dimensional) version
of \eqref{stat}, coming from the expression of Laplacian in polar coordinates. 
We have previously failed to handle this problem in the simpler case
of the power-type nonlinearity considered in \cite{gs}.


\subsection{The prototype}
The function $f$ will be required to satisfy
numerous structural and technical assumptions in order to
establish our results. It may be helpful to keep in mind the following 
typical example.

\begin{example}\label{ex1}
\rm
Under appropriate conditions on $V:\real\to\real$ and $\alpha>0$,
the function $f:\real\x\real_+\to\real$ defined by
\begin{equation}\label{ex.eq}
f(x,s):=V(x)\frac{s^\alpha}{1+s^\alpha}
\end{equation}
will satisfy all of our assumptions. We will state these conditions in due
course, to illustrate the general case.
\end{example} 

\noindent
{\bf Terminology and notation.} For brevity, we will often refer to 
properties of solutions $(\lam,u)$ of \eqref{stat}, e.g. positivity, 
evenness etc., while actually meaning that $u$ possesses these properties.

We will work in both the real and the complex Sobolev spaces 
$H^1(\real,\complex)$ and $H^1(\real,\real)$, depending on whether we consider
\eqref{NLS} or \eqref{stat}, respectively. When no confusion is possible, we
will merely write $H^1(\real)$, and similarly for $L^q(\real)$, $H^2(\real)$,
etc. All of these spaces will be regarded as real Banach spaces, endowed with 
their usual inner products and norms.

The symbol $C$ will denote various positive constants, the exact value of which
does not play an essential role in the analysis.


\section{A global curve of solutions}\label{bifurcation.sec}

Our approach in this section will be based on previous results
\cite{gs,g1,g2} about bifurcation for semilinear equations in $\rn$.
Firstly, in Subsection~\ref{start.sec}, we will show that a {\em local} smooth
branch of solutions of \eqref{stat} bifurcates from the line of trivial 
solutions at the point $(0,0)\in\real\x H^1(\real)$. This will be based
on similar results to those of \cite{g1}, where local bifurcation and
stability results are established for standing waves of the NLS in dimension
$N\ges3$. Under appropriate symmetry and monotonicity assumptions,
we will then show in Subsection~\ref{global.sec} that a version of the 
implicit function theorem can be applied at any positive even solution of 
\eqref{stat}, thereby ensuring global continuation of the local branch. 
The asymptotic bifurcation results of \cite{g3} will then allow us to 
discuss the asymptotic behaviour as $\lam\to\ly$.


\subsection{Bifurcation of small solutions}\label{start.sec}

In \cite{g1} we proved local bifurcation and stability results for the NLS
in dimension $N\ges3$. The nonlinearities we considered in \cite{g1} can be
written as perturbations --- in a sense that will be made more precise 
below --- of the signed-power nonlinearity
\begin{equation}\label{g}
g(x,s):=V(x)|s|^{p-1}s, \quad p>1,
\end{equation}
with $V\in C^1(\rn)$. The main hypotheses about $g$ involve a 
parameter $b\in(0,2)$. Roughly speaking,
it is required that $V(x)\sim |x|^{-b}$
as $|x|\to\infty$ and that the problem be `subcritical', in the sense that
$p<1+\frac{4-2b}{N-2}$. We will formulate the exact hypotheses in the
one-dimensional setting below, but let us already mention two differences from
the case where $N\ges3$. Firstly, if $N=1$, we must impose $b\in(0,1)$. This 
is a requirement of the variational formulation of a limit problem involving
the coefficient $|x|^{-b}$. Secondly, in dimensions $N=1,2$, the problem
is `subcritical' for all $p>1$, so we can dispose of the above upper bound.
However, we will only be interested here in the case of bifurcation
from the line of trivial solutions, while more general situations are
considered in \cite{g1}, allowing for asymptotic bifurcation 
(i.e. bifurcation from $(0,\infty)$ in $\real\times H^1(\rn)$)
as well. This restriction will impose another upper bound on $p>1$, namely
$p<5-2b$. As can be seen from Example~\ref{ex4}, 
this condition is also essential to
the stability of the standing waves of \eqref{NLS}.

We will now state the one-dimensional version of the bifurcation
result of \cite{g1}. 
It is convenient to define
\begin{equation}\label{wtf}
\wt{f}(x,s):=f(x,s^2)s, \quad x,s \in \real.
\end{equation}
We then suppose that
\begin{equation}\label{perturb}
\wt{f}(x,s)=g(x,s)+r(x,s),
\end{equation}
where $g$ is defined in \eqref{g}, $V\in C^1(\real)$ satisfies
\begin{equation}\label{V}
\lim_{|x|\to\infty}|x|^bV(x)=1 \quad\text{and}\quad
\lim_{|x|\to\infty}|x|^bxV'(x)=-b,
\end{equation}
for some $b\in(0,1)$, and
the rest $r$, {\em defined by} \eqref{perturb},
is `small' in a precise, technical sense.
As in \cite{g1}, we are dealing here with situations where the 
linearization of \eqref{stat} at $u=0$ has purely continuous spectrum. 
The method we used in \cite{g1} to
get bifurcation from the continuous spectrum is by perturbation of the model
nonlinearity $g$ that was considered earlier in \cite{gs}. 

A fairly technical 
method was developped in \cite{gs}, based on a rescaling and 
a perturbative argument, using a limit equation involving the nonlinearity
$g$ with $V(x)=|x|^{-b}$. Thus, via continuation from this limit problem,
the asymptotic behaviour of $V$ as $|x|\to\infty$
turns out to govern the local bifurcation from $\lam=0$.
Hypotheses about the rest $r$ are formulated in \cite{g1} ---
see \cite[(r1)-(r5)]{g1} ---,
ensuring that the perturbed nonlinearity retains the main 
properties of $g$ for small $|s|$, and the same asymptotic behaviour
under scaling, in the limit $\lam\to0$. 

In the present context, having 
\eqref{g} to \eqref{V} in mind, 
we will formulate these assumptions in the one-dimensional setting
directly in terms of $\wt{f}$. 
Note that the bifurcation analysis for the model
nonlinearity $g$ was carried out in \cite{g2} in dimension $N=1$, similarly
to the higher dimensional problem treated in \cite{gs}.

Let us finally remark that singularities at $x=0$ were allowed in 
\cite{gs,g1,g2}. We will not need to handle singularities here, and so the
present hypotheses are formulated in a slightly different manner.

\medskip
\begin{itemize}
\item[\bf(A1)] $\wt{f} \in C^1(\real^2)$\footnote{This follows from 
the assumption made in the introduction that 
$f\in C^1(\real\times\real_+,\real)$, but we state it here for completeness.}, 
$\wt{f}(x,\cdot)\in C^2(\real\sm\{0\})$ and 
$\partial_1 \wt{f}(x,\cdot)\in C^1(\real)$
for all $x\in\real$.
\item[\bf(A2)] There exists $b\in(0,1)$, $p\in(1,5-2b)$ and 
$s_0\in(0,\tfrac12]$ such that, \\ for  $0<|s| \les 2s_0$:
$$
|s\partial^2_{22}\wt{f}(x,s)|\les 
\begin{cases}
C |s|^{p-1}, & |x|\les 1,\\
C |x|^{-b}|s|^{p-1}, & |x|\ges 1;
\end{cases}
$$
and
\begin{equation}\label{p1f}
|x\partial^2_{21}\wt{f}(x,s)|\les 
\begin{cases}
C |s|^{p-1}, & |x|\les 1,\\
C |x|^{-b}|s|^{p-1}, & |x|\ges 1.
\end{cases}
\end{equation}
\item[\bf(A3)] Setting $\theta:=(2-b)/(p-1)$, we have
$$
\lim_{k\to0^+}k^{-2}\partial_2\wt{f}\left(\frac{x}{k},k^\theta s\right)=
p|x|^{-b}|s|^{p-1}, \quad\text{for all} \ x\in\real\sm\{0\}, \ s\in\real;
$$
$$
\lim_{k\to0^+}k^{-(2+\theta)}
\left(\frac{x}{k}\right)\partial_1\wt{f}\left(\frac{x}{k},k^\theta s\right)=
-b|x|^{-b}|s|^{p-1}s, \quad\text{for all} \ x\in\real\sm\{0\}, \ s\in\real.
$$
\end{itemize}

\medskip
As usual, we have denoted by $\p_1f$ and $\p_2f$ the partial derivatives of
$f$ with respect to its first and second arguments, respectively.

\medskip
For brevity, we will refer from now on to the collection of assumptions
(A0), (A1), (A2) and (A3) as {\bf (A)}.

\begin{remark}
\rm
Note that the conditions 
$\wt{f}(x,0)=\partial_1\wt{f}(x,0)=\partial_2\wt{f}(x,0)=0, \ x\in\real$,
related to the corresponding conditions on $r$ in \cite[(r1)]{g1},
are automatically satisfied provided that $f$ satisfies (A0).
\end{remark}

\begin{example}\label{ex2}
\rm
The nonlinearity corresponding to the function $f$ defined by \eqref{ex.eq}
is
$$
\wt{f}(x,s)=V(x)\frac{|s|^{2\al}s}{1+|s|^{2\al}}, \quad x,s\in\real.
$$
We suppose that $V\in C^1(\real)$ and that there exists $b\in(0,1)$ such that  
$V$ satisfies \eqref{V}, and we let $\al=(p-1)/2$. 
Then assumption (A) is satisfied if $0<\al<2-b$.
\end{example}

\begin{theorem}\label{local.thm}
Let $\wt{f}$ be defined by \eqref{wtf} and suppose that (A) holds.
There exist $\lam_0>0$ and a function 
$u_0\in C^1((0,\lam_0),H^1(\real))$ such that, for all $\lam\in(0,\lam_0)$,
$(\lam,u_0(\lam))$ is a solution of \eqref{stat} with 
$u_0(\lam)\in C^2(\real)$, $u_0(\lam)>0$ on $\real$, and
$u_0(\lam)(x),u_0(\lam)'(x)\to0$ exponentially as $|x|\to\infty$.
Furthermore, 
\begin{equation}\label{local.eq}
\lim_{\lam\to0}\Ve u_0(\lam)\Ve_{L^2(\real)}=
\lim_{\lam\to0}\Ve u_0'(\lam)\Ve_{L^2(\real)}=
\lim_{\lam\to0}\Ve u_0(\lam)\Ve_{L^\infty(\real)}=0.
\end{equation}
\end{theorem}

\begin{proof} This follows in the same way as Theorem~1~(a) of
\cite{g1}, using the bifurcation results for $N=1$ with the model 
nonlinearity $g$ in \cite{g2}, rather than those of \cite{gs} dealing with
$N\ges3$. Note that the condition $p<5-2b$ in (A2) ensures that
bifurcation occurs from the line of trivial solutions --- this is analogous
to the condition $p<1+\frac{4-2b}{N}$ in Theorem~1 of \cite{g1}.
\end{proof}

\begin{remark}\rm
Note that no symmetry or sign assumptions on the nonlinearity are required for
Theorem~\ref{local.thm}. In fact, the solutions inherit their positivity
from the sign properties of the limit nonlinearity \eqref{g}, and the local
analysis as $\lam\to0$.
\end{remark}


\subsection{Global continuation}\label{global.sec}

We will now prove that, under appropriate assumptions, the local branch of
solutions of \eqref{stat} given by Theorem~\ref{local.thm} can be extended
to a global $C^1$ curve. In particular, we will now suppose that the problem is
symmetric with respect to $x=0$ ---
which will allow us to restrict the discussion
to the half-line, $x\in(0,\infty)$ --- 
and that the nonlinearity satisfies some
monotonicity conditions. 
Our precise hypotheses are the following:

\medskip
\begin{itemize}
\item[\bf(A4)] $f(-x,s)=f(x,s)$ for all $(x,s)\in\real\x\real_+$;
\item[\bf(A5)] $\partial_1 f(x,s)<0$ and $\partial_2 f(x,s)>0$
for all $x,s>0$;
\item[\bf(A6)] (i) $\p_1 f(\cdot,s)\in L^\infty(\real)$ for all $s\ges0$ 
and $\Ve\p_1 f(\cdot,s)\Ve_{L^\infty(\real)}$ is uniformly bounded for $s$ in 
compact subsets of $\real_+$;\\
(ii) $\partial_2 f(\cdot,s)\in L^\infty(\real)$ for all $s\ges0$ and 
$\{\partial_2 f(x,\cdot)\}_{x\in\real}$ is equicontinuous.
\item[\bf(A7)] (AL) holds with 
$f_\infty(0)>\disp\lim_{|x|\to\infty}f_\infty(x)$.
\end{itemize}

For brevity, we will refer from now on to the assumptions (A), (A4) to (A7) 
as assumption {\bf (A')}.

\begin{example}\label{ex3}\rm
In addition to the hypotheses made in Example~\ref{ex2}, we take
$\alpha\ges1$, and we suppose that $V$ is even, $V>0$ on $\real$, 
and $V'(x)<0$ for $x>0$. 
Then the function $f$ defined by \eqref{ex.eq} satisfies
assumption (A').
\end{example}

Let us now collect some important consequences of (A').

\begin{remark} \rm
\item[(a)] From (A0), (A5) and (A7), there exists $M>0$ such that
\begin{equation}\label{fbounded}
0\les f(x,s)\les f_\infty(x)\les M \quad\text{for all} \ (x,s)\in \real_+^2.
\end{equation}
\item[(b)] From (A0), (A4) and (A5), $f_\infty$ is even and 
non-increasing on $[0,\infty)$, with 
\begin{equation}\label{fio.eq}
\lim_{|x|\to\infty}f_\infty(x)=0.
\end{equation}
\item[(c)] (A6)(ii) implies that 
$\Ve\p_2 f(\cdot,s)\Ve_{L^\infty(\real)}$ is uniformly bounded and 
$\{\partial_2 f(x,\cdot)\}_{x\in\real}$ 
is uniformly equicontinuous on the compact subsets of $\real_+$
--- see e.g. \cite[Lemma~5.1]{rs}.
Furthermore, it follows by integration that $\{f(x,\cdot)\}_{x\in\real}$ is 
also uniformly equicontinuous on the compact subsets of $\real_+$. 
\item[(d)] The asymptotic linearization \eqref{lin} 
has a principal eigenvalue. Indeed, setting
$$
-\ly := \inf_{u\in H^1(\real)\sm\{0\}} 
\frac{\intr (u')^2-f_\infty(x)u^2\diff x}{\intr u^2\diff x},
$$
(A7) and \eqref{fio.eq} imply $\ly\in(0,\infty)$. Furthermore,
it follows from the spectral theory of Schr\"odinger operators 
(see e.g. \cite{stuart98}) that $\ly$
is the supremum of the spectrum of \eqref{lin}. Since
$\lim_{|x|\to\infty}f_\infty(x)=0$, we have
$\sigma_\mathrm{ess}=(-\infty,0]$,
where $\sigma_\mathrm{ess}$ denotes the essential spectrum of \eqref{lin}. 
Hence, $\lambda_\infty>0$ is the principal eigenvalue of \eqref{lin}.
\end{remark}

In order to discuss global continuation, 
it is convenient to introduce the function 
$F:\real\x H^1(\real)\to H^{-1}(\real)$ defined by
\begin{equation}\label{F.def}
F(\lam,u)(x):=u''(x)+f(x,u(x)^2)u(x)-\lambda u(x),
\end{equation} 
where $H^{-1}(\real)$ denotes the topological dual of $H^1(\real)$, and the 
right-hand side of \eqref{F.eq} is interpreted as an element of 
$H^{-1}(\real)$ via the canonical identifications:
\begin{equation}\label{ident1}
\la \ffi, v \ra_{H^{-1}\x H^1} \equiv \intr \ffi\, v \diff x
\quad\text{for all} \ \ffi\in L^2(\real), \ v \in H^1(\real);
\end{equation}
\begin{equation}\label{ident2}
\la u'', v \ra_{H^{-1}\x H^1} \equiv -\intr u' v' \diff x
\quad\text{for all} \ u,v \in H^1(\real).
\end{equation}
It follows easily from assumption (A') that 
$F\in C^1(\real\x H^1(\real),H^{-1}(\real))$.
We will obtain positive even solutions of \eqref{stat} by solving the
problem
\begin{equation}\label{F.eq}
F(\lam,u)=0, \quad (\lam,u) \in (0,\infty)\x H^1(\real), \ 
u>0 \ \text{on} \ \real.
\end{equation}
The following result shows that
all solutions of \eqref{F.eq} inherit symmetry and
monotonicity from the nonlinearity.

\begin{lemma}\label{sym.lem}
Let $f$ satisfy (A1), (A4) and (A5), and 
$(\lam,u)\in\real\x H^1(\real)$ be a solution of \eqref{F.eq}.
Then $u(-x)=u(x)$ for all $x\ges0$, $u\in C^3(\real)$ with $u'(0)=0$
and $u'(x)<0$ for all $x>0$, and $(\lam,u)$ is a classical solution of 
\eqref{stat}.
\end{lemma}

\begin{proof}
First, it is easily seen that weak solutions of \eqref{stat} are 
in fact classical solutions, and it follows from (A1) that they are $C^3$.
The remaining statements then 
follow by standard arguments --- see e.g. the proof of \cite[Lemma~2]{js}.
\end{proof}

The following lemma establishes further properties of the solutions,
in particular their exponential decay.
We will suppose that (A') holds throughout the rest of this section.

\begin{lemma}\label{basic.lem}
Let $(\lam,u)\in\real\x H^1(\real)$ be a solution of \eqref{F.eq}.
\item[(i)] $0<\lam<\lam_\infty$.
\item[(ii)] For any $\ep\in(0,\lam)$, let $\eta=\lam-\ep$. Then
there exists $r_\ep>0$ such that
\begin{equation}\label{expu.eq}
|u(x)|\les \Ve u\Ve_{L^\infty(\real)}\e^{-\sqrt{\eta}(|x|-r_\ep)},
\quad\text{for all} \ x \in\real.
\end{equation}
Furthermore, 
\begin{equation}\label{expu'.eq}
\lim_{x\to\pm\infty}\frac{u'(x)}{u(x)}=\pm\sqrt{\lam}.
\end{equation}
\end{lemma}

\begin{proof} To the principal eigenvalue $\lam_\infty$ of \eqref{lin} 
corresponds an eigenfunction $\ffi_\infty>0$. The proof of (i) then follows
in a similar way to that of \cite[Proposition~14~(iv)]{g3}.

Property (ii) was stated in \cite[Proposition~14~(iii)]{g3} in the context
of \eqref{stat} on the half-line, but we did not give the proof explicitly 
there, so we present it here for completeness.
By (A0), for any $\ep\in(0,\lam)$, 
there exists $r_\ep>0$ such that
$$
|x|\ges r_\ep \implies f(x,s)\les \ep <\lam \quad\forall\,s\ges0.
$$
Define a function 
$$
z(x):=\Ve u\Ve_{L^\infty}\e^{-\sqrt{\eta}(|x|-r_\ep)}, \quad x\in\real,
$$
and a set $\Omega_\ep:=\{x\in\real:|x|\ges r_\ep, \ z(x)<0\}$. For all
$x\in\Omega_\ep$ we have
$$
u''(x)=[\lam-f(x,u(x)^2)]u(x)\ges[\lam-\ep]u(x)=\eta u(x).
$$
Hence,
\begin{align*}
z''(x)	&=\eta\Ve u\Ve_{L^\infty}\e^{-\sqrt{\eta}(|x|-r_\ep)}-u''(x)\\
		&\les\eta\big(
		\Ve u\Ve_{L^\infty}\e^{-\sqrt{\eta}(|x|-r_\ep)}-u(x)\big)
		=\eta z(x) \quad\forall\,x\in\Omega_\ep.
\end{align*}
Furthermore, $z(x)=\Ve u\Ve_{L^\infty}-u(x)\ges0$ for $|x|=r_\ep$ and
$\lim_{|x|\to\infty}z(x)=0$. Therefore, if $\Omega_\ep\neq\emptyset$, it
follows by the weak maximum principle \cite[Theorem~8.1]{gt} that $z\ges0$
in $\Omega_\ep$, a contradiction. Hence $\Omega_\ep=\emptyset$ and so
$$
u(x)\les \Ve u\Ve_{L^\infty}\e^{-\sqrt{\eta}(|x|-r_\ep)}, 
\quad |x|\ges r_\ep.
$$
A similar argument applied to $-u$ yields
$$
-u(x)\les \Ve u\Ve_{L^\infty}\e^{-\sqrt{\eta}(|x|-r_\ep)}, 
\quad |x|\ges r_\ep.
$$
Since we clearly have 
$|u(x)|\les \Ve u\Ve_{L^\infty}\e^{-\sqrt{\eta}(|x|-r_\ep)}$ 
for $|x|\les r_\ep$, \eqref{expu.eq} is proved.

Finally, by de l'Hospital's rule,
$$
\lim_{|x|\to\infty}\frac{u'(x)^2}{u(x)^2} =
\lim_{|x|\to\infty}\frac{2u'(x)u''(x)}{2u(x)u'(x)} =
\lim_{|x|\to\infty}\frac{u''(x)}{u(x)} = 
\lim_{|x|\to\infty}\lam-f(x,u(x)^2)=\lam,
$$
where we have used (A0) and (A6)(ii) in the last equality.
Since we know from Lemma~\ref{sym.lem} that $u$ is even with 
$u'<0$ on $(0,\infty)$, \eqref{expu'.eq} follows.
\end{proof}

Let us now prove that the solutions of
\eqref{F.eq} are non-degenerate. This will allow us to extend the local
curve of solutions obtained in Theorem~\ref{local.thm} in a smooth manner.
We will denote by $D_2F$ the Fr\'echet
derivative of $F$ with respect to its second argument.

\begin{lemma}\label{nondegen.lem} 
Let $(\lam,u)\in\real\x H^1(\real)$ be a solution of 
\eqref{F.eq}. Then the linear mapping 
$D_2F(\lam,u):H^1(\real)\to H^{-1}(\real)$ is an isomorphism.
\end{lemma}

\begin{proof} The linear operator 
$D_uF(\lam,u):H^1(\real)\to H^{-1}(\real)$
is explicitly given by 
$$
D_2F(\lam,u)v=v''+[2\p_2f(x,u^2)u^2+f(x,u^2)]v-\lam v.
$$
Note that we can write it as $D_2F(\lam,u)=R_\lam+C$, where 
$R_\lam v:=v''-\lam v$ and $Cv:=[2\p_2f(x,u^2)u^2+f(x,u^2)]v$, 
$v\in H^1(\real)$. Since $u\in H^1(\real)$, it follows from (A0) and (A6)(ii)
that 
\begin{equation}\label{lim.eq}
\lim_{|x|\to\infty}2\p_2f(x,u(x)^2)u(x)^2+f(x,u(x)^2)=0.
\end{equation}
It is then easily seen that $C:H^1(\real)\to H^{-1}(\real)$ is a compact
linear operator. Consequently, since $R_\lam v:H^1(\real)\to H^{-1}(\real)$ 
is an isomorphism for all $\lam>0$, we need only show that $D_uF(\lam,u)$ is
injective. 

To prove this by contradiction, let us suppose that there
exists $v\in H^1(\real)\sm\{0\}$ such that
\begin{equation}\label{v.eq}
v''+[2\p_2f(x,u^2)u^2+f(x,u^2)]v=\lam v \quad\text{in} \ H^{-1}(\real).
\end{equation}
Clearly, it follows that $v\in C^2(\real)\cap H^2(\real)$ 
and that the equation holds in the classical sense. 
We will first prove that $v$ is even. Let
$w\in H^{1}(\real)$ be the odd part of $v$, 
$w(x):=\frac12(v(x)-v(-x)), \ x\in\real$. We need to show that $w\equiv0$.
Suppose instead that $w\not\equiv0$. We have that $w\in C^2(\real)$,
$w(0)=0$ and $w$ is also a solution of the linear equation \eqref{v.eq}. 
Without loss of generality, we can then suppose that $w'(0)>0$. Hence, $w>0$
in a right neighbourhood of $x=0$. Let $x_0>0$ be the first positive zero of
$w$, or $x_0=\infty$ if $w>0$ on $(0,\infty)$. In case $x_0<\infty$ we have
$w(x_0)=0$ and $w'(x_0)<0$.

We now let $z:=u'$. By Lemma~\ref{sym.lem}, $z<0$, $z\in C^2(\real)$, and $z$ 
satisfies
\begin{equation}\label{z.eq}
z''+[2\p_2f(x,u^2)u^2+f(x,u^2)]z+\p_1 f(x,u^2)u =\lam z.
\end{equation}
In case
$x_0<\infty$, integrating the Lagrange
identity for $w$ and $z$ between $x=0$ and $x=x_0$, and using $w(0)=0$, yields
$$
z(x_0)w'(x_0)=\int_0^{x_0}\p_1 f(x,u^2)wu \diff x.
$$
If $x_0=\infty$, we get
$$
0=\lim_{x\to\infty}z(x)w'(x)-w(x)z'(x)=\int_0^\infty\p_1 f(x,u^2)wu \diff x.
$$
It follows from the previous discussion that 
$z(x_0)w'(x_0)>0$. However, in both cases the integral on the right-hand side 
is $<0$ by (A5). This contradiction shows that we must have $w\equiv 0$
indeed. Hence, $v$ is even and $v'(0)=0$.

Consequently, integrating the Lagrange identity for $u$ and $v$ over 
$(0,\infty)$ yields
$$
\intoi\p_2f(x,u^2)u^3v\diff x=0.
$$
Since $\p_2f(x,u(x)^2)u(x)^3>0$ for all $x>0$ by (A5), it follows that
$v$ must have at least one zero in $(0,\infty)$. Furthermore, for any $x>0$,
multiplying \eqref{v.eq} by $v$ and integrating over $(x,\infty)$ yields
$$
\int_x^\infty vv''\diff y + 
\int_x^\infty [2\p_2f(y,u^2)u^2+f(y,u^2)]v^2\diff y =
\lam\int_x^\infty v^2\diff y.
$$
Hence, integrating by parts,
$$
v(x)v'(x)=-\int_x^\infty v'(y)^2\diff y +
\int_x^\infty [2\p_2f(y,u^2)u^2+f(y,u^2)-\lam]v^2\diff y, 
\quad\text{for all} \ x>0. 
$$
But it follows by \eqref{lim.eq} that there exists $r>0$ such that
$$
2\p_2f(y,u(y)^2)u^2+f(y,u(y)^2)-\lam \les -\lam/2
\quad\text{for all} \ y\ges r,
$$
and so $v'(x)v(x)<0$ for all $x\ges r$. In particular, there exists
$x_1>0$ such that $v(x_1)=0$ and $v(x)\neq0$ for all $x>x_1$. Without loss
of generality, we can suppose that $v'(x_1)>0$ and $v(x)>0$ for all $x>x_1$.
Integrating the Lagrange identity for $v$ and $z$ over $(x_1,\infty)$ yields
$$
zv'-vz'\Big\ve_{x_1}^\infty=\int_{x_1}^\infty \p_1 f(x,u^2)vu \diff x.
$$
In view of (A6), it follows easily from \eqref{z.eq} that 
$z\in H^2(\real)$, and so
$$
\lim_{x\to\infty}z(x)v'(x)-v(x)z'(x)=0. 
$$
Therefore,
$-z(x_1)v'(x_1)<0$ by (A5), so that $v'(x_1)<0$. 
This contradiction finishes the proof.
\end{proof}

We are now in a position to prove the main result of this section.

\begin{theorem}\label{global.thm}
Let assumption (A') hold. There exists
$u\in C^1((0,\lam_\infty),H^1(\real))$ such that, for all 
$\lam\in(0,\lam_\infty)$, $(\lam,u(\lam))$ is the unique
positive solution of \eqref{stat},
$u(\lam)\in C^2(\real)\cap H^2(\real)$, $u(\lam)$ is even, and
satisfies Lemmas~\ref{sym.lem}~and~\ref{basic.lem}.

Furthermore, there is bifurcation from the line of trivial solutions at
$\lam=0$, in the sense of \eqref{local.eq}, and
asymptotic bifurcation at $\lam=\lam_\infty$, in the following sense:
if $\lam_n\to\lam\in(0,\lam_\infty]$ as $n\to\infty$ then
\begin{equation}\label{asympt.eq}
\lim_{n\to\infty}\Vert u(\lam_n)\Vert_{H^2(\real)}=
\lim_{n\to\infty} \Ve u(\lam_n) \Ve_{L^\infty(\real)}=\infty \iff \lam=\lam_\infty.
\end{equation}
\end{theorem}

\begin{proof}
Global asymptotic bifurcation for \eqref{stat} on $(0,\infty)$ 
was established in \cite{g3} by degree theoretic arguments. 
Corollary~2 of \cite{g3} holds under the present hypotheses 
--- restricted to the problem on the half-line in the obvious manner. One
needs only remark that, in hypothesis (f2) of \cite{g3}, $f_0$ can be assumed
to be zero, without any change to the proof of Corollary~2. Moreover,
under hypothesis (A5), the assumption 
$\ls:=\limsup_{x\to\infty}f_\infty(x)>f_0$ can be relaxed
to an equality, allowing for \eqref{fio.eq}.
By even extension to $\real$, Corollary~2 of \cite{g3} 
then yields a continuous curve of solutions of
\eqref{F.eq}, $(0,\ly)\ni\lam\to u(\lam)$, 
with the asymptotic behaviour \eqref{asympt.eq}.

Under hypotheses (A4) and (A5), it
can be proved by the method of `separation of graphs', as presented in 
\cite{t}, that the positive solution of \eqref{stat} is unique, for
any $\lam\in(0,\ly)$. Therefore, the local
branch obtained in Theorem~\ref{local.thm} lies on this curve. Thus,
the bifurcation behaviour at $\lam=0$ follows from Theorem~\ref{local.thm},
whereas the asymptotic behaviour as $\lam\to\ly$ follows from
Corollary~2 of \cite{g3}.
Finally, Lemma~\ref{nondegen.lem} shows that 
$u\in C^1((0,\lam_\infty),H^1(\real))$, concluding the proof.
\end{proof}


\section{Orbital stability of standing waves}\label{stability.sec}

Using the function $u\in C^1((0,\lam_\infty),H^1(\real,\real))$ given by
Theorem~\ref{global.thm}, standing wave solutions of
\eqref{NLS} are constructed as
\begin{equation}\label{stwave}
\psi_\lam(t,x):=\e^{i\lam t}u(\lam)(x), \quad \lam\in(0,\lam_\infty).
\end{equation}
The mapping $\lam\to\psi_\lam$ defines a smooth curve of solutions of \eqref{NLS}
in the space $C\big([0,\infty),H^1(\real,\complex)\big) \cap 
C^1\big((0,\infty),H^{-1}(\real,\complex)\big)$. 
General solutions of \eqref{NLS}
are functions $\ffi\in C\big([0,T),H^1(\real,\complex)\big) \cap 
C^1\big((0,T),H^{-1}(\real,\complex)\big)$ satisfying \eqref{NLS} in the weak
sense (with the identifications \eqref{ident1}-\eqref{ident2}), for all
$t\in(0,T)$. Here, $T>0$ determines the maximal interval of existence of the 
solution $\ffi$. If $T=\infty$, the solution is called {\em global} ---
this is obviously the case for standing waves. 

A prerequisite for the stability analysis of \eqref{NLS} is the global
well-posedness of the Cauchy problem. 
This is thoroughly investigated in \cite{caz},
for very general nonlinearities. We will only need the following result here,
which is proved in Section 3.5 of \cite{caz}. 

\begin{theorem}\label{caz.thm}
Let $\wt{f}$ be defined by \eqref{wtf}, with $f\in C^0(\real\times\real_+,\real)$, 
and suppose that there exist $C>0$ and $\sigma\in[0,4)$ such that
\begin{equation}\label{growth}
|\wt{f}(x,s)|\leq C(1+|s|^\sigma)|s|, \quad \text{for all} \ (x,s)\in \real^2.
\end{equation}
For any $\ffi_0\in H^1(\real,\complex)$, there is a unique global solution 
$\ffi\in C\big([0,\infty),H^1(\real,\complex)\big) $ of \eqref{NLS}, 
with initial condition $\ffi(0,\cdot)=\ffi_0$.
\end{theorem}

Under assumption (A'), we can take $\sigma=0$ in \eqref{growth} 
(see \eqref{fbounded}), hence the conclusion of Theorem~\ref{caz.thm} holds.

\begin{definition}\rm
We say that the standing wave $\psi_\lam$ in \eqref{stwave} 
is orbitally stable if
$$
\forall\, \eps>0 \ \exists\, \delta>0 \ \text{such that}
$$
for any (global) solution $\ffi(t,x)$ of \eqref{NLS} with initial data 
$\ffi_0\in H^1(\real)$ we have
$$
\Ve \ffi_0-u(\lam)\Ve_{H^1}\les\delta \implies
\inf_{\theta\in\real}\Ve \ffi(t,\cdot)-\e^{i\theta}u(\lam)\Ve_{H^1}\les\eps
\quad \forall\, t \ges 0.
$$
\end{definition}

A general theory of orbital stability for infinite-dimensional Hamiltonian systems 
was established in \cite{gss} --- see also \cite{stuart2008}, where this issue was 
revisited in great detail, and applied to \eqref{NLS}.  
The stability of a 
standing wave $\psi_\lambda$ is related to spectral properties of the 
linearization of \eqref{NLS} at $\psi_\lambda$. 
When the spectral conditions are 
satisfied, it can be inferred from \cite{stuart2008} that, 
in the present context:\footnote{Note that the parameter $\lam$ in 
\cite[Section~7]{stuart2008} has an opposite sign from ours.}
\begin{equation}\label{slope}
\psi_\lambda(t,x)=\e^{i\lambda t}u_\lambda(x) \ 
\text{is orbitally stable if} \  
\frac{\dif}{\dif\lam}\intr u(\lam)(x)^2 \diff x>0. 
\end{equation}

This condition is often referred to as the {\em slope condition}.

\subsection{The spectral conditions}

Following the discussion in \cite{stuart2008} (see in particular part (5)
of the summary in \cite[Section~7.4]{stuart2008}), the spectral
conditions pertain to the linear operators $L_\lam^1, L_\lam^2: 
H^2(\real,\real)\subset L^2(\real,\real) \to L^2(\real,\real)$ defined by
\begin{align}\label{ops}
L_\lam^1 v  := -v''-\partial_2 \wt{f}(x,u(\lam))v + \lam v,\quad
L_\lam^2 v  := -v''-\frac{\wt{f}(x,u(\lam))}{u(\lam)}v + \lam v.
\end{align}

Let us denote by $M(A)$ the Morse index of a
self-adjoint operator $A:D(A)\subset L^2\to L^2$, defined by
$$
M(A):=\sum_{E \in\mathcal{E}} \dim E ,
$$
where $\mathcal{E}$ is the collection of all eigenspaces corresponding to negative 
eigenvalues of $A$ (and $M(A):=0$ if $\mathcal{E}=\emptyset$). We will also 
denote by $\sigma(A)$ and $\sigma_\mathrm{e}(A)$ the spectrum of $A$ and 
the essential spectrum of $A$, respectively. Then the spectral conditions
required by the stability analysis are the following:

\medskip
\begin{itemize}
\item[\bf(S1)]
$\inf\sigma_\mathrm{e}(L_\lam^1)>0$, $M(L_\lam^1)=1$ and 
$\ker L_\lam^1 =\{0\}$;
\item[\bf(S2)] 
$\inf\sigma_\mathrm{e}(L_\lam^2)>0$, $0=\inf\sigma(L_\lam^2)$  and 
$\ker L_\lam^2=\mathrm{span}\{u(\lam)\}$.
\end{itemize}

\begin{proposition}\label{spectral.prop}
Let (A') hold, $\wt{f}$ be defined by \eqref{wtf}, and
$u\in C^1((0,\lam_\infty),H^1(\real))$ be given by Theorem~\ref{global.thm}.
Then the operators 
$L_\lam^1, L_\lam^2: H^2(\real)\subset L^2(\real) \to L^2(\real)$ 
defined in \eqref{ops} satisfy (S1) and (S2), respectively.
\end{proposition}

\begin{proof}
First of all, an argument similar to the proof of \cite[Lemma~3.4~(i)]{stuart2006} 
shows that all eigenvalues of $L_\lam^1$ and $L_\lam^2$ are simple. 
Also,
$$
\lim_{|x|\to\infty} \partial_2 \wt{f}(x,u(\lam)) =
\lim_{|x|\to\infty} \frac{\wt{f}(x,u(\lam))}{u(\lam)} = 0
\implies  \inf\sigma_\mathrm{e}(L_\lam^i)=\lam>0, \ i=1,2.
$$

We now complete the proof of (S1). First, arguments 
similar to the proof of \cite[Lemma~13]{g1} show that $M(L_\lam^1)=1$
for $\lam>0$ small enough. Since 
$\Vert \partial_2 \wt{f}(x,u(\lam))\Vert_{L^\infty}$ depends continuously
on $\lam\in(0,\ly)$, it follows from the min-max characterization of eigenvalues
(see e.g. \cite[Section~XIII.1]{reed}), that the (isolated) eigenvalues of 
$L_\lam^1$ depend continuously on $\lam\in(0,\ly)$. But $\ker L_\lam^1 =\{0\}$ 
for all $\lam\in(0,\ly)$ by the proof of Lemma~\ref{nondegen.lem}, which prevents
eigenvalues from `crossing zero' as $\lam$ varies.
Hence $M(L_\lam^1)=1$ for all $\lam\in(0,\ly)$.

Regarding (S2), noting that
$$
L_\lam^2 v  = -v''-f(x,u(\lam)^2)v + \lam v, \quad v\in H^2(\real),
$$
it follows from Theorem~\ref{global.thm} that $u(\lam)\in\ker L_\lam^2$.
The eigenvalues of $L_\lam^2$ being simple, 
$\ker L_\lam^2=\mathrm{span}\{u(\lam)\}$, $\lam\in(0,\ly)$. Finally, since
$\inf\sigma_\mathrm{e}(L_\lam^2)>0$,
$0$ is an isolated eigenvalue, with corresponding eigenfunction
$u(\lam)>0$. It follows that $0=\inf\sigma(L_\lam^2)$,
which completes the proof.
\end{proof}


\subsection{The slope condition}\label{slope.sec}

In order to show that the slope condition is verified, 
we will need the following additional assumption.

\medskip
\begin{itemize}
\item[\bf (H)] The function
$$
(x,s) \to \frac{2f(x,s)+x\p_1f(x,s)}{\p_2f(x,s)s}-1, \quad x>0, \ s>0,
$$
is positive and non-increasing as $x$ increases and $s$ decreases.
\end{itemize}

\begin{example}\label{ex4}\rm
Under the hypotheses of Examples~\ref{ex2} and \ref{ex3}, it is easily seen
that the function $f$ defined by \eqref{ex.eq} also satisfies assumption (H),
provided that $-b\les xV'(x)/V(x)<0$ for $x>0$, and that the mapping 
$x\to xV'(x)/V(x)$ is non-increasing for $x>0$.		

Let us now summarize our hypotheses on $V:\real\to\real$ and $\alpha>0$ 
ensuring that the function $f$ defined in Example~\ref{ex1} satisfies (A') and (H):
\begin{itemize}
\item[(a)] $V\in C^1(\real)$ is even, $V>0$ on $\real$, and $V'(x)<0$ for $x>0$;
\item[(b)] there exists $b\in(0,1)$ such that  
$V$ satisfies \eqref{V}; 
\item[(c)] letting $\rho(x):=xV'(x)/V(x)$, we have
that $\rho:(0,\infty)\to\real$ is non-increasing, with $-b\les \rho(x)<0$ 
for all $x>0$;
\item[(d)] $1\les\alpha<2-b$.
\end{itemize}
A function $V$ satisfying properties (a)-(c) is given by 
$$V(x)=(1+x^2)^{-b/2}.$$ 
Note that $\alpha<2-b$ corresponds to $p<5-2b$ in assumption (A2)
(see Example~\ref{ex2}).
Under assumptions (a)-(d), we can take $f$ defined by \eqref{ex.eq} in 
\eqref{NLS} and \eqref{stat}, and the conclusions of Theorem~\ref{global.thm}
and Theorem~\ref{stability.thm} below hold true.

\smallskip
\noindent{\bf N.B.} A more general form of \eqref{ex.eq} can be considered,
e.g. $f(x,s)=V(x)\phi(s)$, with the above assumptions on $V$, and
$\phi\in C^1(\real_+)$ satisfying:
\begin{itemize}
\item[$(\phi 1)$] $\phi(s)>0$ and $\phi'(s)>0$ for $s>0$;
\item[$(\phi 2)$] letting $\Phi(s):=s\phi'(s)/\phi(s)$, there exists
$\alpha:=\lim_{s\to0^+}\Phi(s)\in[1,2-b)$, 
$\Phi:(0,\infty)\to\real$ is non-increasing, 
and $0<\Phi(s)\les \alpha$ for all $s>0$;
\item[$(\phi 3)$] $\disp \lim_{s\to\infty}\phi(s)=1$.
\end{itemize}
In particular, conditions (c) and $(\phi 2)$ 
(simply (c) and (d) in the special case \eqref{ex.eq})
ensure that (H) is satisfied.
\end{example}

In view of Proposition~\ref{spectral.prop} and \cite{stuart2008}, the orbital
stability of the standing waves $\psi_\lam$ in \eqref{stwave} will be
established if we prove the following result.

\begin{proposition}\label{slope.prop}
Let hypotheses (A') and (H) hold, 
and $u\in C^1((0,\lam_\infty),H^1(\real))$ 
be given by Theorem~\ref{global.thm}. Then
\begin{equation}\label{slop}
\frac{\dif}{\dif\lam}\intr u(\lam)(x)^2 \diff x>0
\end{equation}
for all $\lam\in(0,\ly)$.
\end{proposition}

Our proof of Proposition~\ref{slope.prop} is very similar to that of (1.5)
in \cite[Theorem~1.7]{g2}, 
where a power nonlinearity is considered. Since the structure
of the nonlinearity is quite different in the present case, one needs to check
carefully that every step of the proof carried out in Section~5.2 of \cite{g2}
works under assumption (A'). Therefore, we will present the whole argument here.

We will suppose that assumption (A') holds for the rest of this section. We 
already know from Theorem~\ref{global.thm} --- more precisely from bifurcation
at $\lam=0$ --- that there is some $\lam>0$ for which \eqref{slop} holds.
Therefore, by continuity, we need only show that 
\begin{equation}\label{nonzero}
\frac{\dif}{\dif\lam} \intr u(\lam)(x)^2 \diff x
\equiv 2 \frac{\dif}{\dif\lam} \intoi u(\lam)(x)^2 \diff x  \neq0
\quad \forall\,\lam\in(0,\ly). 
\end{equation}
Now,
$$
\frac{\dif}{\dif\lam} \intoi u(\lam)(x)^2 \diff x 
= 2\intoi u(\lam)(x)\xi(\lam)(x)\diff x,
$$
where
$$
\xi(\lam):=\frac{\dif u}{\dif\lam}(\lam)\in H^1(\real) ,
\quad \text{for all} \ \lam\in(0,\ly).
$$
We will use a fairly involved integral indentity derived from the equations 
satisfied by $u(\lam)$ and $\xi(\lam)$, namely \eqref{stat}
and
\begin{equation}\label{equforxi}
\xi'' + f(x,u^2)\xi + 2\p_2 f(x,u^2) u^2 \xi = \lam \xi + u.
\end{equation}
(We will omit the variables $\lam$ and/or $x$ when no confusion is possible.)
The equation for $\xi$ is easily obtained by differentiation of the identity 
$F(\lam,u(\lam))=0$ with respect to $\lam$, where $F$ is defined in \eqref{F.def}. 
We will first prove two lemmas establishing the required integral identity and 
some useful properties of $\xi$. 
Proposition~\ref{slope.prop} will then be proved using these results. 

The strategy of proof applied here was first used in 
\cite{mst}, where the authors considered a similar problem, however with a 
power-type nonlinearity (although their proof allows for more general situations
--- albeit not the asymptotically linear case),
and a non-trivial linear potential.

\begin{lemma}\label{intid.lem}
For $u,\xi\in H^1(\real)$ as above, the following identity holds:
\begin{equation}\label{intid}
\intoi [2f(x,u^2)+x\p_1f(x,u^2)-\p_2f(x,u^2)u^2]u\xi\diff x
=2\lam\intoi u\xi\diff x.
\end{equation}
\end{lemma}

\begin{proof}
The proof follows that of \cite[Lemma~5.2]{g2}.
Let us first remark that, using the properties of $u$, it follows from
\eqref{equforxi} that $\xi\in C^2(\real)\cap H^2(\real)$, and $\xi$ is
even with $\xi'(0)=0$. Integrating the Lagrange identity for \eqref{stat}
and \eqref{equforxi} then yields
\begin{equation}\label{intid1}
\intoi u^2 \diff x=2\intoi \p_2f(x,u^2)u^3\xi \diff x .
\end{equation}
On the other hand, multiplying \eqref{stat} by $u$ and integrating gives
\begin{equation}\label{intid2}
\intoi f(x,u^2)u^2-(u')^2\diff x = \lam\intoi u^2 \diff x.
\end{equation}
Now multiplying \eqref{stat} by $xu'$ and integrating by parts 
(using the exponential decay of $u'$) yields
$$
\intoi xf(x,u^2)uu'-(u')^2-xu'u'' \diff x = \lam\intoi xuu'\diff x.
$$
Using \eqref{stat} and integrating by parts, it follows that
\begin{align}\label{tempid1}
\intoi 2xf(x,u^2)uu'-(u')^2\diff x
& = 2 \lam \intoi xuu'\diff x 
= 2 \lam \intoi x\big(\tfrac12 u^2\big)'\diff x\notag\\
& = \lam\Big[xu^2\Big\ve_0^\infty-\intoi u^2 \diff x\Big] 
= -\lam\intoi u^2 \diff x.
\end{align}
Furthermore, computing
$$
\frac{\dif}{\dif x}\Big(x\int_0^{u^2}f(x,s)\diff s\Big)=
\int_0^{u^2} f(x,s)\diff s + x f(x,u^2)2uu' + x\int_0^{u^2}\p_1f(x,s)\diff s,
$$
we can substitute the first term of the LHS of \eqref{tempid1} and we get
\begin{multline}\label{tempid2}
\intoi 
\Big\{\frac{\dif}{\dif x}\Big(x\int_0^{u^2}f(x,s)\diff s\Big)
-\int_0^{u^2}[ f(x,s) + x\p_1f(x,s)]\diff s -(u')^2\Big\}\diff x\\
= -\lam\intoi u^2 \diff x.
\end{multline}
But the first term in the LHS of \eqref{tempid2} can be integrated and yields
$$
x\int_0^{u^2}f(x,s)\diff s \, \Big|_0^\infty=
\lim_{x\to\infty}x\int_0^{u^2}f(x,s)\diff s=0
$$
since, by \eqref{fbounded} and the exponential decay of $u$, 
$$
\Big|x\int_0^{u(x)^2}f(x,s)\diff s\Big|\les M x u(x)^2\to0
\quad \text{as} \ x\to\infty.
$$
Hence we finally have
\begin{equation}\label{intid3}
\intoi 
\int_0^{u^2} [f(x,s) + x\p_1f(x,s)]\diff s + (u')^2\diff x
= \lam\intoi u^2 \diff x.
\end{equation}
Now adding \eqref{intid2} and \eqref{intid3} yields
\begin{equation}\label{intid4}
\intoi 
f(x,u^2)u^2+\int_0^{u^2} [f(x,s) + x\p_1f(x,s)]\diff s \diff x
= 2\lam\intoi u^2 \diff x.
\end{equation}
Next, differentiating \eqref{intid4} with respect to $\lam$, we obtain
\begin{equation}\label{tempid4}
\intoi [2f(x,u^2)+x\p_1f(x,u^2)+\p_2f(x,u^2)u^2]u\xi\diff x
=\intoi u^2 \diff x + 2\lam\intoi u\xi\diff x.
\end{equation}
Finally, using \eqref{intid1} to substitute the first term of the RHS of 
\eqref{tempid4}, we get \eqref{intid}.
\end{proof}

The following lemma establishes useful properties of $\xi$.

\begin{lemma}\label{propofxi}
For all $\lam\in(0,\ly)$, 
the function $\xi=\xi(\lam)$ has the following properties: $\xi(0)>0$ and there exists a unique $x_0=x_0(\lam)\in(0,\infty)$ such that $\xi(x_0)=0$, $\xi(x)\ges0$ for all $x\in(0,x_0)$, and $\xi(x)\les0$ for all $x\in(x_0,\infty)$.
\end{lemma}

\begin{proof}
We fix
$\lam\in(0,\ly)$, and we simply write $\xi(x)$ for $\xi(\lam)(x), \ x\in\real$.
The first part of the proof ---
showing that $\xi(0)>0$ --- is easily adapted from that of \cite[Lemma~5.3]{g2},
using the sign of $\p_1f(x,u^2)$ given by (A5). 

To prove the existence of
an $x_0\in(0,\infty)$ ($x_0=x_0(\lam)$) with the asserted properties, 
it is convenient to establish first
that there is no $a>0$ such that $\xi(x)\ges0$ for all 
$x\ges a$. This is the part of the proof where some amendments to
that of \cite[Lemma~5.3]{g2} need to be mentionned.
We will prove this by contradiction, so we assume that such an $a$ 
exists. 
Integrating the Lagrange identity for \eqref{z.eq} and \eqref{equforxi}
from $x\ges a$ to $\infty$ yields
$$
-z(x)^2\left(\frac{\xi}{z}\right)'(x)=\xi(x)z'(x)-\xi'(x)z(x)
=\int_x^\infty uz+\p_1f(y,u^2)u\xi \diff y<0, \  x\ges a.
$$
Hence $(\xi/z)'>0$ on $[a,\infty)$ and $\xi/z$ is increasing on this interval. But $\xi/z\les0$ on $[a,\infty)$ and so there exists a number $L$ such that
$$
\lim_{x\to\infty}\frac{\xi(x)}{z(x)}=L\les0.
$$
Since $\xi\in H^2(\real)$, we have $\xi'(x)\to0$ as $x\to\infty$. 
Furthermore, $z'(x)\to0$ as $x\to\infty$ by \eqref{stat}.
It then follows from de l'Hospital's rule that
$$
L=\lim_{x\to\infty}\frac{\xi(x)}{z(x)}=\lim_{x\to\infty}\frac{\xi'(x)}{z'(x)}
=\lim_{x\to\infty}\frac{\xi''(x)}{z''(x)},
$$
as long as the last limit exists. Now
\begin{align*}
\frac{\xi''}{z''}
&=\frac{\lambda\xi+u-f(x,u^2)\xi-2\p_2f(x,u^2)u^2\xi}
{\lambda z-f(x,u^2)z-\p_1f(x,u^2)u-2\p_2f(x,u^2)u^2z}\\
&=\frac{\lambda\xi/z+u/z-f(x,u^2)\xi/z-2\p_2f(x,u^2)u^2\xi/z}
{\lambda-f(x,u^2)-\p_1f(x,u^2)u/z-2\p_2f(x,u^2)u^2},
\end{align*}
where, as $x\to\infty$:  
$$
\xi/z\to L, \quad u/z\to -\lambda^{-1/2},
$$
$$
f(x,u^2)\to0 \ \text{by (A0)}, 
\quad \p_2f(x,u^2)u^2\to0  \ \text{by (A6)(ii)},
$$
and
$$
\p_1f(x,u^2)\to0  \ \text{by \eqref{p1f}}.
$$
Hence
$$
L=\lim_{x\to\infty}\frac{\xi''(x)}{z''(x)}
=\frac{\lambda L-\lambda^{-1/2}}{\lambda}
=L-\lambda^{-3/2},
$$
a contradiction. Therefore, there is no $a>0$ such that $\xi(x)\ges0$ for all 
$x\ges a$. The remainder of the proof is easily adapted from that of
\cite[Lemma~5.3]{g2}, using again $\p_1f(x,u^2)<0$. We leave the details to the 
reader.
\end{proof}

\medskip

\noindent{\it Proof of Proposition~\ref{slope.prop}.}
We prove \eqref{nonzero} by contradiction. Hence we suppose that 
$\intoi u(\lam)(x)\xi(\lam)(x)\diff x=0$ for some $\lam\in(0,\ly)$.
We omit the dependence on $\lam$ from $u$ and $\xi$ for the rest of the proof.
It follows from \eqref{intid} and our assumption that
\begin{equation}\label{intidzero}
\intoi 
\biggl\{\frac{2f(x,u^2)+x\p_1 f(x,u^2)}{\p_2f(x,u^2)u^2} - 1\biggr\}
\p_2f(x,u^2)u^3\xi \diff x=0.
\end{equation}
Letting 
$$
\zeta(x):=\frac{2f(x,u^2)+x\p_1 f(x,u^2)}{\p_2f(x,u^2)u^2} - 1,
\quad x>0,
$$
it follows from (H) that $\zeta$ is positive and non-increasing on $(0,\infty)$.
Using the zero $x_0$ of $\xi$ given by Lemma~\ref{propofxi}, we can rewrite
the identity \eqref{intidzero} as
\begin{equation*}
\intoi[\zeta(x)-\zeta(x_0)]\p_2f(x,u^2)u^3\xi \diff x
+\zeta(x_0)\intoi\p_2f(x,u^2)u^3\xi \diff x=0.
\end{equation*}
By \eqref{intid1}, this becomes
\begin{equation*}\label{intidzero'}
\intoi[\zeta(x)-\zeta(x_0)]\p_2f(x,u^2)u^3\xi \diff x
+\frac{\zeta(x_0)}{2}\intoi u^2 \diff x=0.
\end{equation*}
Since $\intoi u^2 \diff x>0$, the properties of $\zeta$ yield the desired
contradiction. The proposition is proved.\hfill$\Box$

\begin{theorem}\label{stability.thm}
Suppose that hypotheses (A') and (H) hold. The standing waves of \eqref{NLS} 
defined by \eqref{stwave} are orbitally stable.
\end{theorem}

\begin{proof}
In view of the discussion of orbital stability for \eqref{NLS} in 
\cite{stuart2008} (see part (5)
of the summary in Section~7.4 of \cite{stuart2008}), the result follows
immediately from Propositions~\ref{spectral.prop} and \ref{slope.prop}.
\end{proof}


\section{Self-focusing planar waveguides}\label{physics}

In this last section we will briefly present an important application of our 
results to nonlinear optics. Self-focusing planar waveguides have been 
thoroughly investigated, both from the physical and the mathematical 
standpoints. Amongst the wide literature about nonlinear waveguides, let us mention 
\cite{saleh,akhmanov,svelto,chiao,stegseg,vakh} regarding the 
physics, and \cite{stuart93,stuart2007,sz96,sz05,g2,weinfib} 
for some mathematical results.
Historically, the mathematical study of \eqref{NLS} has grown in 
parallel to --- and was largely motivated by --- the development of 
nonlinear waveguide theory. A striking example of this close interaction
is the early paper \cite{vakh}, where the slope condition \eqref{slope}
was first introduced to discuss the stability of travelling waves in
a cylindrical waveguide with a saturable dielectric response. (We will
precisely be interested in saturable materials here, see 
\eqref{saturable} below.)

The mathematical modelling of a planar self-focusing waveguide is summarized in
\cite{g2}, where we used bifurcation and stability results similar to
Theorem~\ref{global.thm} and Theorem~\ref{stability.thm} to discuss the
behaviour of TE travelling waves in a waveguide with a power-type
dielectric response. 
The planar waveguide is idealized as a slab of dielectric material parallel to the
$xz$-plane, having infinite extension in the $x$-direction, and 
semi-infinite extension along $z$, in the $z>0$ direction. 
We will look for electromagnetic waves in the optical regime, travelling
along the $z>0$ half-axis, and so $x$ will be transverse to the 
direction of propagation.
According to Maxwell's equations, the propagation of a light beam depends on
the material, characterized by a dielectric response $\eps>0$ 
(or, alternatively, the refractive index $n=\sqrt{\eps}$). 
In a nonlinear medium, the 
dielectric response can be decomposed as
$$
\eps(x,s)=\eps_L(x)+\eps_{NL}(x,s),
$$
where $\eps_L(x)$ and $\eps_{NL}(x,s)$ respectively denote the linear and 
nonlinear contributions to the dielectric response. As we shall see below,
the variable $s$ is proportional to the squared modulus 
of the electric field of the light beam.
The dependence on the variable $x$ accounts for a medium which is
inhomogeneous in the transverse direction. We will consider the case where the
linear contribution, defined by $\eps_L\equiv\eps(x,0)$, is a positive constant, 
i.e. the material is homogeneous when the beam is switched off. On the other hand,
we will suppose that $\eps_{NL}(-x,s)\equiv \eps_{NL}(x,s)$, and $\eps_{NL}$
decreases away from $x=0$. This behaviour helps 
to focus the waves around $x=0$ since,
according to Snell's law, the light beam bends towards regions
with a higher refractive index.
A dielectric medium is called {\em self-focusing} when $\eps_{NL}(x,s)$
is an increasing function of the variable $s$ (hence of the beam's intensity).

In \cite{g2} we treated the case of a {\em Kerr medium}, where 
\begin{equation}\label{Kerr}
\eps_{NL}(x,s)=a(x)s
\end{equation} 
for some positive even function $a\in C^1(\real)$, decreasing for $x>0$.
Assumption \eqref{Kerr}, leading to a cubic nonlinearity,
gives a good approximation, in the low power regime, 
of so-called `Kerr materials'. However, other materials 
--- e.g. photorefractive materials (see \cite{c}) --- present a {\em saturation}
phenomenon as the power of the beam becomes large. Namely,
\begin{equation}\label{saturable}
\eps_{NL}(x,s)\to \eps_\infty(x) \quad \text{as} \  s\to\infty,
\end{equation}
where $\eps_\infty(x)$ adds up to $\eps_L$, yielding the asymptotic
dielectric response $\eps_L+\eps_\infty(x)$, in the limit of high power beams.
As can be seen from the equations below, this leads to an
asymptotically linear nonlinearity (sometimes referred to as a 
`saturable nonlinearity'). We are now able to deal with this case.

A TE travelling wave is a special solution of Maxwell's equations in the 
waveguide, having an electric field transverse to the $xz$-plane, of the form
\begin{equation}\label{tetwave}
E(x,y,z,t)=\re\big(0,U(x)\e^{i(kz-\omega t)},0\big).
\end{equation}
The enveloppe $U$ of the electric field is such that 
$U\in H^1(\real,\real)$ --- so as to verify the `guidance
conditions', namely
decay of the electromagnetic field at infinity and
finiteness of the energy density ---, and satisfies the Helmholtz equation
\begin{equation}\label{helm2}
U''+\textstyle(\frac{\omega}{c})^2
[\eps_L+\eps_{NL}(x,\frac12 U^2)]U=k^2 U, \quad x\in\real.
\end{equation}
Since we only consider monochromatic
waves, i.e. with $\omega>0$ a fixed frequency in the optical regime,
the time variable $t$ in \eqref{tetwave} 
does not play an essential role in the analysis. 
The stability of the TE travelling waves \eqref{tetwave} is usually discussed
in the context of the {\em paraxial approximation}, with respect to the
larger class of TE modes, 
which are more general solutions of Maxwell's equations than \eqref{tetwave} 
(still with an electric field transverse to the $xz$-plane).
After a rescaling of the variables, this approximation leads to
the following equation, governing the behaviour of TE modes in the waveguide:
\begin{equation}\label{paraxial2}
i\p_z\psi+\p^2_{xx}\psi +
\textstyle(\frac{\omega}{c})^2\eps_{NL}(x,\frac12|\psi|^2)\psi=0,
\quad z>0, \ x\in\real.
\end{equation} 
A detailed discussion of
the paraxial approximation is given in \cite{g2}, where we also explain
the reduction of the problem to the canonical form \eqref{paraxial2}, from
which the constant $\eps_L$ has disappeared. In this approach,
the TE travelling waves \eqref{tetwave} can be approximated by 
standing wave solutions of \eqref{paraxial2}. Since \eqref{paraxial2}
has the form of \eqref{NLS}, and \eqref{helm2} that of \eqref{stat},
we can apply the results of the previous sections,
yielding existence of solutions for \eqref{helm2}, hence of TE travelling
waves, and orbital stability of standing waves of \eqref{paraxial2}.
As explained in \cite[Section~6.3]{g2}, a global bifurcation result
such as Theorem~\ref{global.thm} enables
one to make the paraxial approximation as accurate as desired. Hence
stability of the travelling waves \eqref{tetwave} can be inferred from
Theorem~\ref{stability.thm} --- the physical meaning of `stability'
for the TE travelling waves \eqref{tetwave} is discussed in Section~6.2 of
\cite{stuart2007}, where Kerr media are studied using the results in
\cite{js,mst,stuart2006}.

The rigorous proof of stability of the travelling waves 
\eqref{tetwave}, thus obtained from the analysis of \eqref{NLS},
puts on firm mathematical grounds the phenomenon of {\em self-trapping}, known
to physicists since the early 70's 
(first predicted theoretically and later observed experimentally --- 
see \cite{stegseg}). This phenomenon is briefly described
as follows.
In a self-focusing medium, the nonlinear contribution $\eps_{NL}(x,s)$ 
is an increasing function of $s>0$, typically larger near to $x=0$. 
By increasing locally, around $x=0$,
the dielectric response, the light beam induces its own waveguide,
forcing the waves to focus in the $x=0$ region, along the $z$-axis. When this
focusing effect balances the dispersion due to the term $\p^2_{xx}\psi$
in \eqref{paraxial2}, self-trapping occurs, yielding stable guided waves.

Our main results for self-focusing planar waveguides with a saturable dielectric 
response \eqref{saturable} follow from a similar discussion to that in 
\cite[Section~6.3]{g2} for the Kerr medium \eqref{Kerr}. They can be 
summarized as follows.
Letting 
$$
f(x,s)=\textstyle(\frac{\omega}{c})^2\eps_{NL}(x,\frac12s),
$$
we suppose that assumptions (A') and (H) are satisfied.
Then, applying Theorem~\ref{global.thm} and Theorem~\ref{stability.thm}
to \eqref{helm2} and \eqref{paraxial2}, we get the following results.

\begin{itemize}
\item[(1)] There exist guided TE travelling waves \eqref{tetwave},
corresponding to solutions $U=U_k$ of \eqref{helm2},
for all wave numbers $k\in(k_1,k_3)$, where
$$
k_1:=\frac{\omega}{c}\sqrt{\eps_L} \quad\text{and}\quad
k_3:=\big(\textstyle\left(\frac{\omega}{c}\right)^2\eps_L+\ly\big)^{1/2}.
$$
\item[(2)] The TE travelling waves \eqref{tetwave} are stable, 
for all $k\in(k_1,k_3)$.
\item[(3)] Defining the power of the beam associated with $U_k$ as
$$P(k):=\frac{c^2k}{2\omega}\intr U_k(x)^2 \diff x,$$ 
it follows from our bifurcation
results in Theorem~\ref{global.thm} that 
$$
\lim_{k\to k_1}P(k)=0 \quad \text{and} \quad \lim_{k\to k_3}P(k)=\infty.
$$
Hence, we obtain guided waves from arbitrary low to arbitrary high power.
\end{itemize}

\end{document}